\documentclass[12pt,a4paper,oneside]{amsart}

\ifdefined\SMART
\usepackage[paperwidth=12cm,paperheight=24cm,left=5mm,right=5mm,top=10mm,bottom=5mm]{geometry}
\else
\calclayout
\fi

\usepackage{color,amssymb,latexsym,amsfonts,textcomp,fancyhdr,calc,graphicx}
\usepackage{pdfpages,amsfonts,amsthm,amssymb,latexsym,graphicx,leftidx,fancyhdr}
\usepackage{amsmath,amstext,amsthm,amssymb,amsxtra}
\usepackage[usenames,dvipsnames,svgnames,table,rgb]{xcolor}
\usepackage[allcolors=blue,allbordercolors=blue,pdfborderstyle={/S/U/W 1}]{hyperref}
\usepackage[ocgcolorlinks]{ocgx2}
\usepackage{dsfont}
\usepackage[framemethod=tikz]{mdframed}
\usepackage{enumerate}
\usepackage[normalem]{ulem}

\usepackage{graphicx}
\usepackage{txfonts,pxfonts,tikz} %also txfonts 
\usepackage{tgschola}
\usepackage[T1]{fontenc}
\usepackage[framemethod=tikz]{mdframed}
\DeclareMathOperator{\rank}{rank}
\usepackage{mathtools}

\newtheorem{theorem}{Theorem}[section]

\theoremstyle{definition}

\newcommand{\beql}[1]{\begin{equation}\label{#1}}
	\newcommand{\eeq}{\end{equation}}
\newcommand{\comment}[1]{}
\newcommand{\Ds}{\displaystyle}

\newcommand{\Abs}[1]{{\left|{#1}\right|}}

\newcommand{\Set}[1]{{\left\{{#1}\right\}}}

\newcommand{\RR}{{\mathbb R}}

\newcommand{\ZZ}{{\mathbb Z}}
\newcommand{\NN}{{\mathbb N}}

\newcommand{\QQ}{{\mathbb Q}}

\newcommand{\one}{{\bf 1}}

\newcommand{\vol}{{\rm vol\,}}

\newcommand{\wt}[1]{\widetilde{#1}}

\newcommand{\red}[1]{#1}%\textcolor{red}{#1}}
\newcommand{\blue}[1]{#1}%\textcolor{blue}{#1}}
\newcommand{\trn}{\Set{0}^m\times\RR^n}%{\Set{0}^m\hspace{-0.3em}\times\RR^n}
\newcommand{\ztrn}{\ZZ^m\times\RR^n}%{\ZZ^m\hspace{-0.3em}\times\RR^n}

\newcounter{rem}
\newcounter{step}
\newcounter{mysec}
\newcounter{mysubsec}[mysec]
\title[Bounded lattices tiles that pack]{Bounded lattice tiles that pack with another lattice}

\author{Sigrid Grepstad}
\address{\href{https://www.ntnu.edu/imf}{Department of Mathematical Sciences}, Norwegian University of Science and Technology (NTNU), NO-7491, Trondheim, Norway.}
\email{sigrid.grepstad@ntnu.no}
\thanks{S.G.\ is supported by Grant 334466 of the Research Council of Norway.}

\author{Mihail N. Kolountzakis}
\address{\href{http://math.uoc.gr/en/index.html}{Department of Mathematics and Applied Mathematics}, University of Crete,\\Voutes Campus, 70013 Heraklion, Greece,\\and\\ \href{https://ics.forth.gr/}{Institute of Computer Science}, Foundation of Research and Technology Hellas, N. Plastira 100, Vassilika Vouton, 700 13, Heraklion, Greece}
\email{kolount@gmail.com}

\author{Emmanuil Spyridakis}
\address{\href{http://math.uoc.gr/en/index.html}{Department of Mathematics and Applied Mathematics}, University of Crete,\\Voutes Campus, 70013 Heraklion, Greece.}
\email{manos.ch.spyridakis@gmail.com}
\thanks{E.S.\ gratefully acknowledges the support of the ERC starting grant 101078061 SINGinGR,
under the European Union's Horizon Europe program for research and innovation}

\begin{document}

% Temporary hack until amsart gets updated to include 2020 Classification
\makeatletter
\@namedef{subjclassname@2020}{\textup{2020} Mathematics Subject Classification}
\makeatother


\begin{abstract}

Suppose $L, M$ are full-rank lattices in Euclidean space, such that $\vol L < \vol M$. Answering a question of Han and Wang \cite{han2001lattice} from 2001, we show how to construct a bounded measurable set $F$ (we can even take $F$ to be a finite union of polytopes) such that $F+L$ is a tiling and $F+M$ is a packing.
In the related problem of finding a common fundamental domain for two lattices $L$ and $M$ we show that no such common fundamental domain exists if the volumes are different even, if we do not require measurability of this domain.

\end{abstract}

\date{\today}

\maketitle

\tableofcontents

\section{Introduction}

\subsection{The concept of a lattice and its fundamental domains}
A  lattice $L$ in $\RR^d$ is a discrete subgroup of $\RR^d$. The dimension of the $\RR$-linear subspace that is spanned by $L$ is called the \textit{rank} of $L$. In particular if $L$ spans $\RR^d$, we will call it a full-rank lattice.
It is known that any full-rank lattice on $\RR^d$ is equal to $A\ZZ^d$ for some nonsingular $d\times d$ matrix. A \textit{fundamental parallelepiped} of the lattice $A\ZZ^d$ is the set $P=A[0,1)^d$. Observe this is not uniquely defined as each lattice $L$ admits many matrices $A$ such that $L=A\ZZ^d$. The volume of $P$ is called the volume of the lattice and it is denoted by $\vol(L)$. The volume of the lattice does not depend on which matrix $A$ we use in the representation of the lattice as $A\ZZ^d$.
	
For a given lattice $L\subseteq \RR^d$, any fundamental parallelepiped $P$ of $L$ has the property of containing exactly one element from each class of the quotient group $\RR^d/L$ (each coset of $L$ in $\RR^d$). $P$ is not the only set with this property. In fact there are many sets with this property and each of them is called a \textit{fundamental domain} for the lattice $L$.  

\subsection{Lattice tilings}
Suppose $\Omega\subseteq \RR^d$ is a measurable set and $A\subseteq \RR^d$ is a discrete set. We say that $\Omega$ \textit{packs} $\RR^d$ by translations by $A$, if the following inequality holds:
\begin{equation}\label{pack}
	\sum_{t\in A} \one_{\Omega}(x+t)\le 1
\end{equation} 
for almost every $x\in \RR^d$. If instead of inequality in (\ref{pack}) we have equality for almost every $x\in \RR^d$, then we say that $\Omega$ \textit{tiles} $\RR^d$ by translations of the set $A$. 

It is easy to see that a fundamental parallelepiped $P$ of a full-rank lattice $L\subseteq \RR^d$, forms a tiling of $\RR^d$ by translations with the lattice. Since $P$ contains exactly one point from each class in $\RR^d/L$, we get that $x$ can be uniquely written as an element of $P$ that belongs to the class $x \bmod(L)$, translated by a point in the lattice $L$. In fact every fundamental domain of $L$ forms a tiling of $\RR^d$ by translations with $L$, while the converse is also true with a minor adjustment: every measurable set that tiles $\RR^d$ by translations with $L$, differs from a fundamental domain of $L$ by a null set (a set of measure $0$). We call such sets \textit{almost fundamental domains} of $L$.

By what we said previously, any measurable subset of an almost fundamental domain  for a lattice $L\subseteq \RR^d$, packs $\RR^d$ by translations of $L$. For the converse we have that a set which packs $\RR^d$ by translations of $L$, is a subset of a fundamental domain up to a null set.

\subsection{Common fundamental domains}
The study of the fundamental domains for a lattice is related to the well-known \textit{Steinhaus tiling problem} \cite{sierpinski1958probleme}, in which  it is asked, whether there exists a set $E\subseteq \RR^2$ that tiles $\RR^2$ by translations of any rotation of the lattice $\ZZ^2$. In other words we are seeking a common fundamental domain for all lattices of the form $R\ZZ^2$, for $R$ being a $2\times2$ rotation matrix in $\RR^2$. This problem has a set-theoretic and a measurable incarnation. In the first case we seek a precise (no exceptions), not necessarily measurable fundamental domain for all these lattices (this was solved in the affirmative in \cite{jackson2002sets}). In the second case we seek a measurable set that is an almost fundamental domain for each of these lattices (this is still open in the plane and the best results so far are in \cite{kolountzakis1999steinhaus}). We are mostly interested in the measurable version of the problem.

A sensible relaxation of this problem is to seek a common almost fundamental domain for a finite family of lattices. These lattices must all have the same volume for such  a set to exist. In \cite{kolountzakis1997multi} it was proved that, whenever we have finitely many full rank lattices on $\RR^d$ with the same volume, such that their dual lattices have a direct sum \blue{(i.e., their dual lattices intersect trivially)}, there exists a measurable common fundamental domain for all the lattices in that finite family. This set is generally unbounded. Under the assumption that the sum of the lattices themselves is direct it was also proved that there exists a not necessarily measurable but \textit{bounded} common fundamental domain for all the lattices in that finite family. Later Han and Wang in \cite{han2001lattice} dropped the density assumption for the sum of two lattices and  showed the existence of a measurable, in general not bounded, common fundamental domain for any two lattices of the same volume. They also showed that this result cannot be generalized for more than two lattices with the same volume, by giving a counterexample (as was done in \cite{kolountzakis1997multi}).

Recently, Grepstad and Kolountzakis \cite{GK2025} proved the existence of a measurable bounded common almost fundamental domain for any two lattices with the same volume.
	
\subsection{Bounded common fundamental domain requires equal volumes even without measurability}
A question that comes naturally from what we mentioned above, is whether these results can be generalized for any two full rank lattices. In particular, can we have a common fundamental domain for two lattices of different volume? 

In the measurable case, it is obvious that such a set cannot exist since every tile set of a lattice has measure that equals the volume of the lattice. However, if we do not demand measurability it is unclear whether such a set exists. In \cite[Theorem 1]{kolountzakis1997multi} it is shown, among others, that if two lattices in $\RR^d$ have a direct sum then a common fundamental domain of them in $\RR^d$ exists, but, of course, is not measurable if the volumes are different.

It is also shown in \cite{kolountzakis1997multi} that if two lattices in $\RR^d$ have a direct sum and the same volume then they have a common, bounded fundamental domain, which is not necessarily measurable. It is interesting in this result that the volume of the lattices seems to play a role in the properties of the common fundamental domain despite not requiring measurability of the domain.

We will show here that in the case where the sum of the two lattices is direct (i.e., they intersect only at the origin) and the lattices have unequal volumes, then any common fundamental domain must be an unbounded set.
In other words we will prove:
\begin{theorem}\label{NoFD}
Assume that $L,M$ are two-full rank lattices in $\RR^d$, with $\vol(L)< \vol(M)$ such that $L\cap M =\{0\}$. Furthermore assume that $F$ is a common fundamental domain of $L, M$ in $\RR^d$. Then $F$ is unbounded.
\end{theorem}

Theorem \ref{NoFD} extends the corresponding result that was given in \cite[Theorem 3.3]{kolountzakis2022functions} which concerned lattices that are dilates of each other.

\subsection{Tiling with one lattice, packing with the other}
In the case where the volumes of two lattices $L,M$ differ, say $\vol(L)<\vol(M)$, Han and Wang \cite[Theorem 1.2]{han2001lattice} showed the existence of a measurable set  $F$ that tiles $\RR^d$ by translations with the lattice $L$ and packs  $\RR^d$ by translations with the lattice $M$. Han and Wang raised the question whether this set $F$ can always be chosen to be bounded. Our main result shows that such an $F$ can indeed always be taken to be a bounded set.
\begin{theorem}\label{Main}
Let $L, M$ be two full rank lattices on $\RR^d$, such that $\vol(L)<\vol(M)$. There exists a bounded measurable set $F$ that tiles $\RR^d$ by translations with $L$ and packs $\RR^d$ by translations with $M$.
\end{theorem}

In different language Theorem \ref{Main} says that we can find a bounded almost fundamental domain of the sparser lattice $M$ which contains an almost fundamental domain of the denser lattice $L$. One might think, especially in light of the recent result in \cite{GK2025} (any two lattices of equal volume have a bounded common almost fundamental domain), that Theorem \ref{Main} might be proved by dilating the lattice $L$ to make it have equal volume with $M$, taking a common fundamental domain of these two and pass to a subset of it that is a fundamental domain for $L$.

This plan however is not feasible: there are fundamental domains of the dilated lattice which contain no fundamental domain of the original lattice. The easiest way to see this is to consider in $\RR$ the lattice $L = \ZZ$ and the lattice $M = \alpha\ZZ$, where $\alpha>1$ is irrational. Using the density of $M \bmod L$ it is easy to construct, for any $\epsilon>0$, a fundamental domain of $M$ which, taken mod 1,  lands in $[0, \epsilon]$, so no subset of it can be a fundamental domain of $L$.

In \S \ref{noFD}  we will prove Theorem \ref{NoFD}, and in \S \ref{MR} we will prove our main result, Theorem \ref{Main}.

\section{No bounded Common Fundamental Domain {if volumes differ}}\label{noFD}
In this section we will prove Theorem \ref{NoFD}. We should point out that in \cite{kolountzakis2022functions} it was shown that given the lattices $\ZZ^d$ and $a\ZZ^d$ where $a$ is an irrational number there does not exist a bounded common fundamental domain of them in $\RR^d$. Theorem\ref{NoFD} generalizes this result.
\begin{proof}[\textbf{Proof of Theorem \ref{NoFD}}]
Without loss of generality we may assume that $L=\ZZ^d$ and $M= A\ZZ^d$, for some $A$, $d\times d$ matrix with $|\det(A)|>1$. Assume that $F$, a common fundamental domain of $L$ and $M$, is bounded (no measurability is assumed). It is enough for us to show that there is no bounded common fundamental domain $K$ of the lattices $L, M$ in the group $L+M$. The reason for this is that if $F$ were bounded then we could define $K$, also bounded, as
$$
K = F \cap (L+M).
$$

We can write $L+M= \{ n + Am:\  n,m \in \ZZ^d\}$ and all the sums $n+A m$ are distinct because of our assumption $L\cap M = \Set{0}$.
\blue{
This assumption also implies that every common fundamental domain of $L,M$ in $L+M$ must be of the form $\{l_j +m_j \:|\: j\in \NN \}$ for some enumeration of the elements of $L$ and $M$.
}
Thus we can write $K=\{n +Am_n:\ n,m_n \in \ZZ^d\}$ for any common fundamental domain of $L$ and $M$ in $L+M$ (the $m_n$ are a permutation of $\ZZ^d$). Assume, in order to arrive at a contradiction,  that $K \subseteq B_r := [-r, r]^d$, for some constant $r>0$. Take $R \to +\infty$ and observe that, since $K\subseteq B_r $, if $n \in B_R$ then we get that:
\begin{equation}\label{Am_n}
	 	A m_n \in B_{R+r}.
\end{equation}
Therefore the number of different $m_n$ corresponding to all the $n \in B_R$ is
\begin{equation}\label{Rr}
\lesssim \frac{|B_{R+r}|}{\det A} = \frac{2^d(R+r)^d}{\det(A)}.
\end{equation}
As $R\to \infty$ the number of values of $n\in \ZZ^d\cap B_R$ is $\approx 2^d R^d$. But to each of the $n \in B_R$ corresponds a unique $m_n$, so the number in \eqref{Rr} should be at least $2^dR^d$, which it fails to be if $R$ is sufficiently large.
\end{proof}

\blue{
While this paper was under consideration, we found an alternative proof of Theorem \ref{NoFD}. It relies on the nontrivial tool of Banach measures, but requires no additional assumptions on the lattices involved. Theorem 1.1 may thus be strengthened as follows.
}

\begin{theorem}\label{th:different-volumes}
If $L, M$ are lattices in $\RR^d$ with different volumes then they do not have a bounded common fundamental domain (even a non-measurable one).
\end{theorem}

\begin{proof}
The proof is based on the existence of {\em Banach measures}. These are translation-invariant finitely additive set-functions on all subsets of $\RR^d$ which agree with Lebesgue measure on all Lebesgue measurable sets. Banach measures are known to exist since the time of Banach \cite[Corollary 12.9]{tomkowicz2016banach}.

Let $\mu$ be a Banach measure and assume that $F\subseteq \RR^d$ is a bounded fundamental domain of $L$. Let also $Q=Q_L$ be the fundamental parallelepiped of $L$. We will show that $\mu(F) = \mu(Q) = \Abs{Q} = \det L$. This proves the theorem.

Consider the tilings of $\RR^d$ by the $L$-translates of $Q$ and $F$. Then we have $Q = \bigcup_{l \in L} Q \cap (F+l)$ and notice this is a finite union since both $Q$ and $F$ are bounded. It follows that
$$
\mu(Q) = \mu\left(\bigcup_{l\in L} Q \cap (F+l)\right) = \sum_{l \in L} \mu(Q \cap (F+l)),
$$
by the finite additivity of $\mu$.

Similarly, we also have $F = \bigcup_{l \in L} F \cap (Q+l)$ (again, a finite union), hence
\begin{align*}
\mu(F) &= \mu\left(\bigcup_{l \in L} F \cap (Q+l)\right)\\
 &= \sum_{l \in L} \mu(F \cap (Q+l))\\
 &= \sum_{l \in L} \mu(F \cap (Q-l))\ \ \text{change ($l$ to $-l$)}\\
 &= \sum_{l \in L} \mu((F+l) \cap Q)\ \ \text{(translation invariance)}\\
 &= \mu(Q).
\end{align*}
\end{proof}

\section{Proof of Theorem \ref{Main}}\label{MR}
\red{We use induction on $d$. If $d=1$ then $F = [0, \vol L]$ tiles with $L$ and packs with $M$ since $F \subseteq [0, \vol M]$. For what follows we will assume that Theorem \ref{Main} is true for all dimensions smaller than $d$. Induction will be needed only in the proof of Theorem \ref{ZxR} below.}

The closed subgroups of $\RR^d$ are, up to a nonsingular linear transformation, of the form
\begin{equation}\label{L+M}
\ZZ^m \times \RR^n
\end{equation}
where $m+n=d$, where $m=0,1,...,d$ (see Theorem 9.11 of \cite{hewitt2012abstract}). Thus we may assume that $\overline{L+M}=\ZZ^m \times \RR^n$ for some such decomposition $d=m+n$. We treat separately the cases: $m=0$, $m=d$ and $m\in \{1,2,..d-1\}$.

\subsection{Dense on $\RR^d$ {($m=0$)}}

This is the key case and it is proved by a careful implementation of the method first given in \cite{kolountzakis1997multi}.
\begin{theorem}\label{R}
Let $L,M$ be two full-rank lattices on $\RR^d$, such that $\vol{(L)}<\vol{(M)}$ and $\overline{L+M}=\RR^d$. There exists a bounded measurable set $F$ that tiles $\RR^d$ by translations of $L$ and packs $\RR^d$ by translations of $M$. In particular, $F$ can be taken as a finite union of polytopes.
\end{theorem}
\begin{proof}
Without loss of generality we can assume that $L=\ZZ^d$ and $M=A\ZZ^d$, where \blue{$A$ is a $d\times d$ invertible matrix} with $|\det(A)|>1$ and  $\overline {\ZZ^d + A\ZZ^d}=\RR^d$. Denote by $P_1 = [0, 1)^d$, a fundamental parallelepiped of $L$, and let $P_2$ be a fundamental parallelepiped of the lattice $M$.

The idea of the proof is simple: We partition the $d$-dimensional unit cube $P_1$, into axis aligned, disjoint in measure cubes of the same length and for $P_2$ we find axis-aligned, disjoint in measure cubes of the same but slightly larger length so that each of them lies inside $P_2$ (here we don't assume a partition of $P_2$). Then, using the density assumption of the group $L+M$, we can translate by an element of $L+M$ every cube  in the partition of $P_1$ \textit{inside} a (different each time) translated by an element of $M$, cube in $P_2$. This can be done since the cubes in $P_2$ are larger than the cubes in the partition in $P_1$. The fact that we can find more cubes in $P_2$ than the ones in the partition of $P_1$ comes straight from our assumption: $|\det(M)|>1$. Finally the union of these suitable translated cubes of $P_1$ by elements of $L$ is our set $F$.

\begin{figure}\label{construction on the plane}
\ifdefined\SMART\includegraphics[width=150mm]{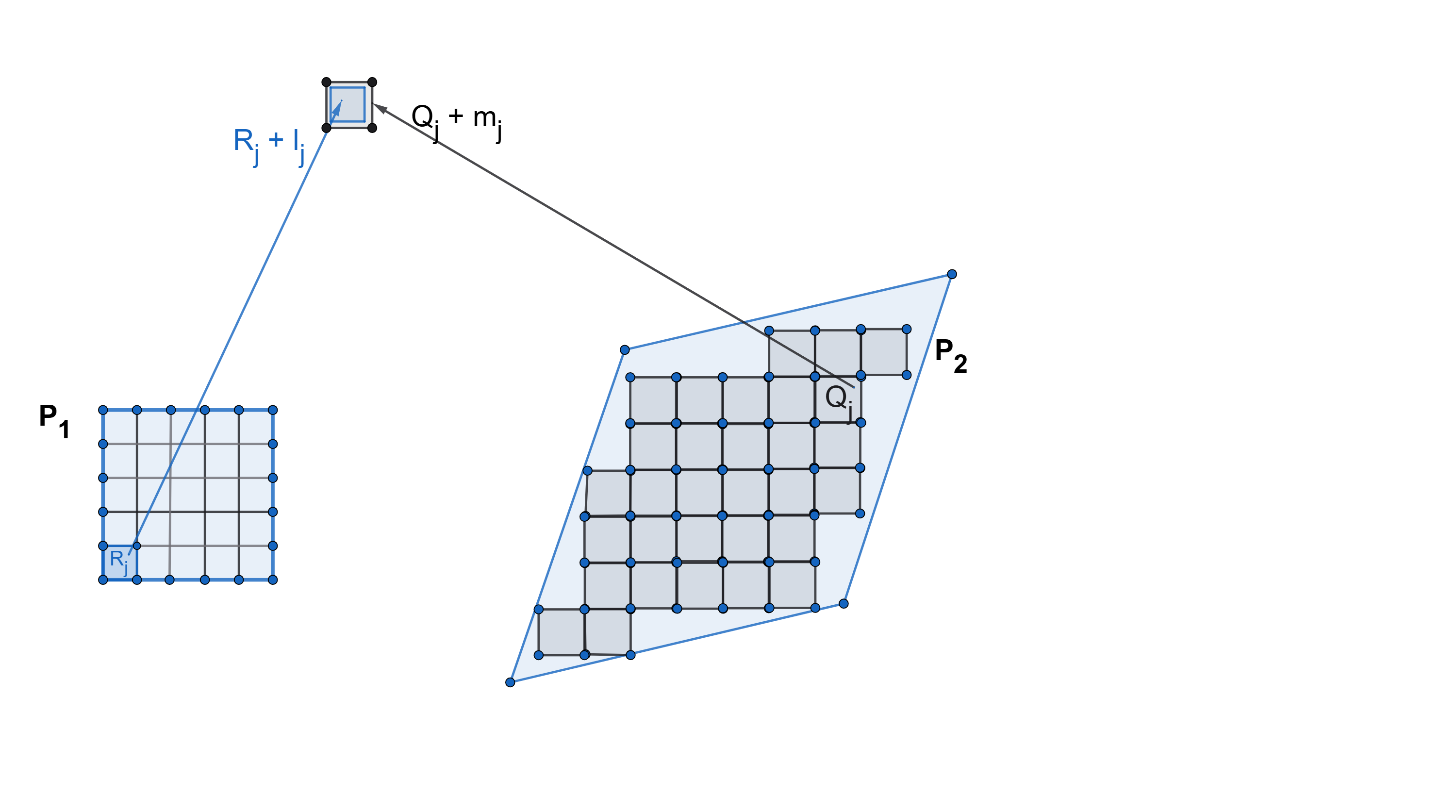}\else
\includegraphics[width=200mm]{Construction.png}
\fi
\caption{Construction of the set $F$ on $\RR^2$. The letters $R_j$, $Q_j$ represent cubes in $P_1$, $P_2$ respectively, while $l_j \in L$, $m_j \in M$.}
\end{figure}

We start by finding these cubes for $P_1$ and $P_2$. Pick a number $c^{\frac{1}{d}} \in (1,|\det(M)|^{\frac{1}{d}})$ and find a $N\in \NN$ such that, when we impose a grid of side-length ${c^\frac{1}{d}}/{N}$ on $P_2$ and keep only the whole $({c^\frac{1}{d}}/{N})$--cubes, the leftover measure of $P_2$ is  $<|\det(M)|-c$. If $N$ is sufficiently large then it is clear that this can be achieved, as the leftover set is contained in a $C/N$-neighborhood of the boundary of $P_2$ so their total volume is at most $C'/N$ and we arrange this to be however small we want if we take $N$ to be large.

This implies that the number of cubes of volume $c N^{-d}$ is $> N^d$. For such an $N$ call these finitely many axis-aligned and disjoint in measure cubes that lie inside $P_2$, $Q_s$ with $s\in \{1,2,\ldots,S\}$, where $S > N^d$. Similarly call the axis-aligned disjoint in measure cubes of length $\frac{1}{N}$ that partition $P_1=[0,1)^d$, $R_s$, for $1\le s\le N^d$.

Our set $F$ will be of the form
\begin{equation}\label{defF}
F = \bigcup_{j\le N^d} \left( R_j + l_j\right),
\end{equation}
for appropriately chosen $l_j \in L$. This set $F$ clearly tiles with $L$ for any choice of the $l_j$ (notice also that all sets $R_j+l_j$ are disjoint). Since $L+M$ is dense we can find $l_j+m_j\in L+M$ such that $R_j+l_j+m_j$ is contained in the (slightly larger) cube $Q_j$, for $j=1, 2, \ldots, N^d$. This process defined the $l_j$. Notice also that
$$
F \subseteq \bigcup_{j \le N^d} \left(Q_j - m_j\right),
$$
and the latter union is identical $\bmod (M)$ to a subset of $P_2$. This means that $F+M$ is a packing as we had to show.
\end{proof}

\subsection{The fully commensurable case ($m=d$)}
Here we treat the case where $m=d$ in (\ref{L+M}). For this case, the theorem below has been proved by Han and Wang \cite{han2001lattice}.
\blue{This is the easiest case of our main result, and it is commonly known. For completeness we include an alternative proof.}
\begin{theorem}\label{Z}
Let $L,M$ be two full rank lattices on $\RR^d$, such that $\vol(L) \le \vol(M)$ and ${L+M}= \ZZ^d$. There exists a bounded measurable set $F$ that tiles $\RR^d$ by translations of $L$ and packs $\RR^d$ by translations of $M$. Again $F$ can be taken as a finite union of polytopes.
\end{theorem}
\begin{proof}
Clearly we have that $\vol(L),\vol(M)\in \ZZ_{+}$. We shall first find a finite set $F_1$ that tiles $L+M=\ZZ^d$ by translations of $L$ and packs $\ZZ^d$ by translations of $M$. Then we define
$$
F = [0, 1]^d + F_1,
$$
which is obviously a bounded and measurable set and satisfies the requirements of the Theorem.

Let $H=L\cap M$, $G=\ZZ^d/H$, $L_1=L/H$ and $M_1=M/H$ and notice that it is enough for us to find a set $E$ that tiles $G$ by translations of $L_1$ and packs $G$ by translations of $M_1$. Then, since the group $H$ in the quotient groups $G,M_1,L_1$ is the same, we can take $F_1:=E$. Keep in mind that $F_1$ must contain exactly one element from every class $\bmod(L)$ in $\ZZ^d$, while the set $E$ contains one element from each class in $\Ds\frac{\ZZ^d/H}{L/H}=\ZZ^d/L$. Furthermore, each element of $F_1$ must belong to a different class $\bmod(M)$ while $E$ contains at most one element from each class in $\Ds\frac{\ZZ^d/H}{M/H} = \ZZ^d/M$.

 Write $l=[G : L_1]$ and $m=[G:M_1]$. From $\vol(L)\le\vol(M)$, we clearly have that $l\le m$. Write the elements of each quotient group as $G/L_1 = \{g_1,...,g_l\}$ and $G/M_1 =\{h_1,...,h_m\}$. By a slight abuse of notation, we can view the elements $g_i,h_i$ as elements of $G$, even though they are classes in the quotient groups $G/L_1, G/M_1$, by picking arbitrary representatives for each class in the quotient groups. Since the sum $L_1+M_1=G$ is direct ($L_1 \cap M_1 = \Set{0}$), we can further assume that $g_i \in M_1$ for every $i\le l$, and $h_k \in L_1$ for every $k\le m$. Finally let $E = \{ g_i +h_i:\  i\le l\}$. Notice that $E$ contains exactly one element from each class $\bmod(L_1)$ which means $L_1+E = G$, that is, it tiles $G$ by translations of $L_1$. Also each element of $E$ belongs to a different class $\bmod(M_1)$ and so $F_1$ packs $G$ by translations of $M_1$, as we had to prove.

\end{proof}

\subsection{The intermediate case ($m\in \{1,2,...,d-1\}$)}
Finally we show the last case in (\ref{L+M})  of our main result Theorem \ref{Main}, 
which completes our proof. Some notation that we will use for the proof below: whenever we have two subgroups of a group $G$, call them $H_1,H_2$, such that $H_1 \cap H_2=\Set{0}$ we will denote their sum as $H_1 \oplus H_2$ instead of $H_1 + H_2$ (we call this a \textit{direct sum}). Also, for two sets $A,B$ that don't have the structure of a group, we will denote their sum as $A\oplus B$ only in the case where the translations by $B$ copies of $A$ do not overlap. In this case if $C=A\oplus B$, we say that $A$ is a tile of $C$ by translations of $B$ and symmetrically $B$ is a tile of $C$ by translations of $A$.
\begin{theorem}\label{ZxR}
Let $L,M$ be two full rank lattices on $\RR^d$, such that $\vol(L)< \vol(M)$ and $\overline{L+M}=\ZZ^m \times \RR^n$, where $m\in \{1,2,...,d-1\}$ and $n=d-m$. \red{Then there} exists a bounded measurable set $F$ that tiles $\RR^d$ by translations of $L$, and packs $\RR^d$ by translations of $M$.
\end{theorem}
\begin{proof}
	Once again, it is enough to find a bounded measurable set $F_1$ that tiles $\overline{L+M}= \ZZ^m \times \RR^n$ by translations of $L$ and packs $\ZZ^m \times \RR^n$ by translations of $M$. Then the set 
	\begin{equation*}
		F:= [0,1]^m \times \{0\}^n + F_1,
	\end{equation*}
    which is clearly a bounded and measurable set, satisfies the requirements of our theorem. In particular we will show that $F_1$ is a finite union of polytopes and, therefore, so is $F$.
    
A brief description of what will follow: we will first split each of our lattices $L$, $M$ into two disjoint sublattices $L_1,L_2$ and $M_1,M_2$ respectively such that $L=L_1\oplus  L_2$ and $M=M_1\oplus M_2$. The sublattices $L_2,M_2$ will lie in the subspace $\{0\}^m \times \RR^n$ having rank $n$, while $L_1,M_1$, which lie in $\ZZ^m \times \RR^n$, have rank $m$. Then we will show \red{that}
\begin{equation}\label{L_1+M_1}
   	L_1 +M_1 +\{0\}^m\times \RR^n = \ZZ^m \times \RR^n.
\end{equation}

\red{Since $L_2, M_2 \subseteq \{0\}^m \times \RR^n$ and $n<d$ we can by induction} find a set $E\subseteq \{0\}^m \times \RR^n$ that tiles $\{0\}^m \times \RR^n$ by translations of $L_2$ and packs $\{0\}^m \times \RR^n$ by translations of $M_2$. Of course, in order to use
\red{the inductive hypothesis}, 
we also want to have $\vol_n(L_2)\le \vol_n(M_2)$, where $\vol_n (.)$ is simply the $n$-dimensional volume. We achieve this by considering an appropriate superlattice for $L_2$ and $M_2$ on $\{0\}^m \times \RR^n$ with the desired volumes. Finally, we will find a \red{finite} set of  suitable, disjoint, translated copies of the set $E$ which will tile/pack all slices $\{k\} \times \RR^n$ for $k\in \ZZ^m$ by translations of $L$ and $M$ respectively. This set will be the set $F_1$.

We start by defining the lattices $L_2= L \cap (\{0\}^m \times \RR^n)$ and $M_2= M\cap (\{0\}^m \times \RR^n)$. We will show that $L_2$ has rank $n$ and so has $M_2$. Indeed if $\rank(L_2)=r<n$ then   there exists a sublattice $L_1$ of $L$, of rank $d-r>m$, such that $L=L_1 \oplus L_2$ \cite[Corollary 3, p. 14]{cassels1996introduction}.
Let $l_1,...,l_{m+1}$ be $m+1$, $\RR$-linearly independent elements of $L_1$. Denote by $\pi (l_i)$  the projection of $l_i$ to the space $\ZZ^m \times \{0\}^n$ and so $\pi(l_1)$, say, is a \red{$\QQ$}-linear combination of $\pi(l_2),...,\pi(l_{m+1})$ which\red{, after clearing denominators,} implies that there exists a non-zero element of $L_1$ which is a $\ZZ$-linear combination of $l_1,...,l_{m+1}$, that lies on $\{0\}^m \times \RR^n\cap L= L_2$ which contradicts with the fact that $L_1 \cap L_2=\{0\}$. In fact we showed something stronger: given a $\ZZ$ basis of $L_1$, $\{l_1,..,l_m\}$, its projection to $\ZZ^m \times \{0\}^n$ is a set consisting of $m$, $\ZZ$-linearly independent elements on $\ZZ^m \times \{0\}^n$. Equivalently every element in $L_1$ must belong to a different class $\bmod (\{0\}^m \times \RR^n)$.

Abusing notation we can write $L = L\ZZ^d$, $M = M\ZZ^d$, where $L, M$ are $d \times d$ nonsingular matrices. The columns of these matrices can be any basis of the lattices
so we choose the first $m$ to be a basis of $L_1$ (resp.\ $M_1$) and the last $n$ to be a basis of
$L_2$ (resp.\ $M_2$). The matrices $L, M$ are now lower block triangular
\begin{equation*} L=
	\begin{pmatrix}
		L_1 & 0 \\
		\star & L_2
	\end{pmatrix},
	\qquad
	M=
	\begin{pmatrix}
		M_1 & 0 \\
		\star & M_2
	\end{pmatrix} 
\end{equation*}
where the $m\times m$ matrices $L_1, M_1$ have integer entries since these entries represent
the first $m$ coordinates of a basis of $L_1, M_1 \subseteq \ZZ^m \times \RR^n$. It follows that $\vol(L)= \det (L_1) \det(L_2)$ and $\vol(M)=\det(M_1)\det(M_2)$. To avoid confusion, whenever $\det(L_i)$ or $\det(M_i)$ is written, we refer to $L_i$ or $M_i$ as the diagonal {blocks of the matrices $L$ and $M$}. In every other case whenever $L_i$ or $M_i$ appears in the text, we mean the sublattice $L_i$ or $M_i$ respectively for $i = 1,2$.

\red{Summarizing we have}
\begin{equation}\label{L_1+R}
[\ZZ^m \times \RR^n : L_1 \oplus \{0\}^m \times \RR^n]= \red{\det L_1},
\end{equation}
\red{and}
\begin{equation}\label{M_1+R}
[\ZZ^m \times \RR^n : M_1 \oplus \{0\}^m \times \RR^n]= \red{\det M_1}.
\end{equation}
    
To obtain \eqref{L_1+M_1} observe that
\begin{align*}
\ztrn &\supseteq L_1+M_1+\trn \\
 &= \overline{L_1+M_1+\trn} \\
 &\supseteq \overline{L_1+M_1+L_2+M_2}\\
 &= \overline{L+M}\\
 &= \ztrn.
\end{align*}

Now that we are done with the splitting of the lattices $L$ and $M$, we want to apply
\red{the inductive hypothesis}
for the lattices $L_2, M_2$, but we do not know if $\det(L_2) \le \det(M_2)$. For this reason we consider superlattices $L_2', M_2'$ of $L_2, M_2$ respectively inside $\{0\}^m \times \RR^n$ such that $[L_2': L_2]=\det(M_1)$ and $[M_2':M_2]= \det(L_1)$. Clearly by our assumption that
$$
\det(L_1)\det(L_2)=\vol(L)< \vol(M)=\det(M_1)\det(M_2)
$$
we get that
$$
\det(L_2 ')= \frac{\det(L_2)}{\det(M_1)}<\frac{\det{(M_2)}}{\det{(L_1)}}=\det(\red{M_2'}).
$$
\red{Since $L_2', M_2' \subseteq \trn$ and $n<d$ we obtain by induction}
a bounded measurable set $E$ (finite union of polytopes) of volume
\begin{equation}\label{volE}
\vol_n(E) = \det(L_2 ')= \frac{\det(L_2)}{\det(M_1)},
\end{equation}
that tiles $\{0\}^m \times \RR^n$ by translations with $L_2 '$ and packs $\{0\}^m \times \RR^n$ by translations with $M_2 '$.

Finally, in order to find the set $F_1$ we need to find the suitable set of copies of the set $E$ that we mentioned at the start of the proof. For this, let us first find a finite set $J_2\subseteq L_2 '$ and $K_2\subseteq M_2 '$ such that $L_2 ' =L_2 \oplus J_2$ and $M_2 '= M_2 \oplus K_2$. The set $J_2$ (and similarly  $K_2$) can be obtained by taking a complete set of coset representatives of $L_2$ in $L_2'$ and of $M_2$ in $M_2'$. Since $[L_2 ':L_2] = \det(M_1)$ we get $|J_2|= \det(M_1)$ and similarly $|K_2| = \det(L_1)$.

\begin{figure}
\ifdefined\SMART\includegraphics[width=110mm]{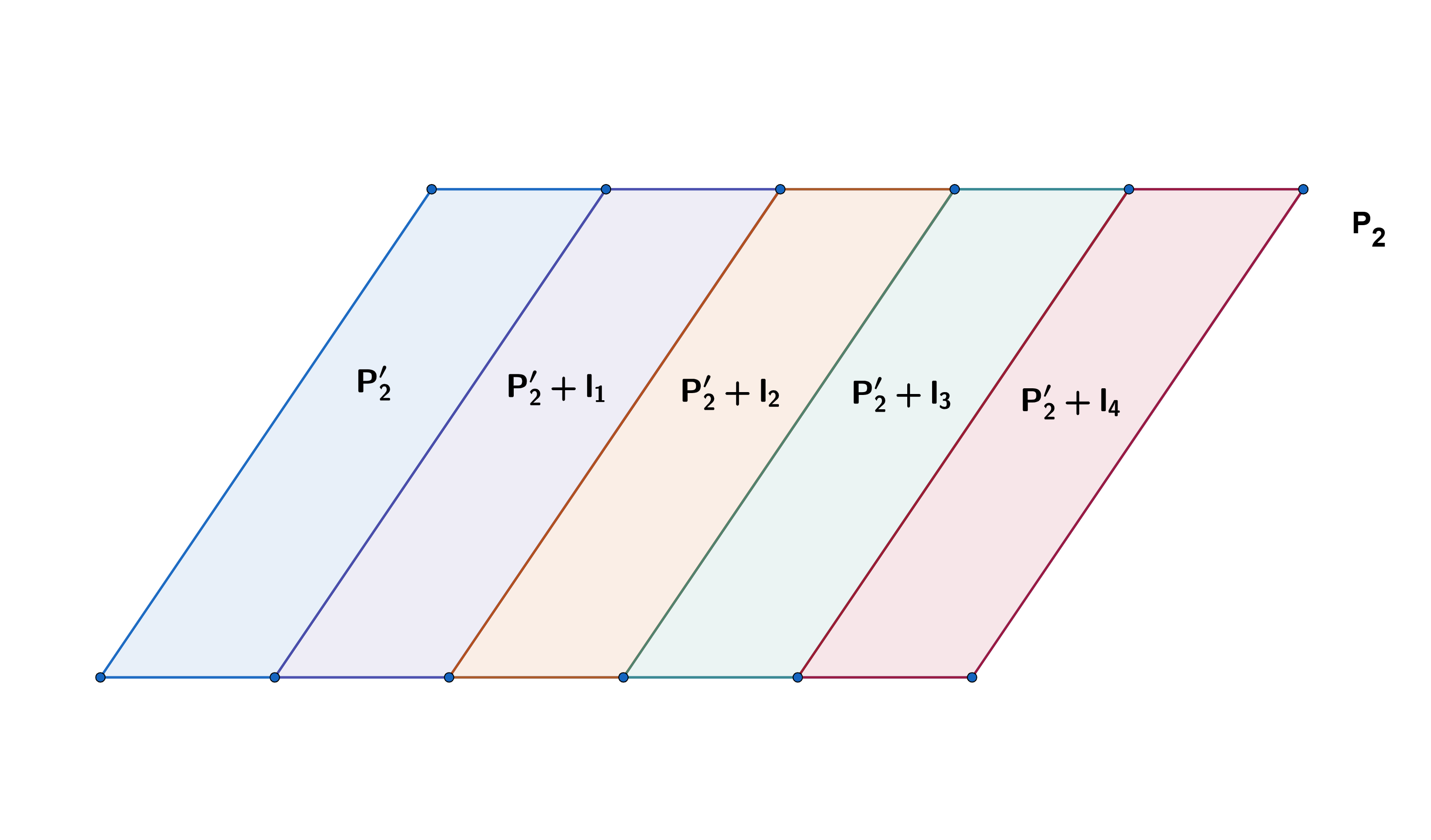}\else
   	\includegraphics[width=150mm]{PFP}
\fi
   	\caption{Obtaining the set $J_2$ when $L_2'$ is a superlattice in the plane. $P_2$, $P_2 '$ are the fundamental parallelepipeds of $L_2$, $L_2 '$ with $L_2 \subseteq L_2 '$ where $[L_2' : L_2]=5$. $J_2=\{0,l_1,l_2,l_3,l_4\}$}.
\end{figure}

Now, regarding the lattices $L_1,M_1$, from (\ref{L_1+M_1}) we can get a finite set $J_1 \subseteq L_1$ and a finite set $K_1 \subseteq M_1$ such that 
\begin{equation}\label{J_1+M_1}
   	J_1 \oplus M_1 \oplus \{0\}^m \times \RR^n =\ZZ^m \times \RR^n
\end{equation} and 
\begin{equation}\label{K_1 + L_1}
   	K_1 \oplus L_1 \oplus \{0\}^m \times \RR^n =\ZZ^m \times \RR^n
\end{equation}
From ($\ref{L_1+R}$) we get that  $|K_1| = \det(L_1)=|K_2|$ and from  (\ref{M_1+R}) we get that $|J_1|=\det(M_1)=|J_2|$.
Take any two bijections \red{$\psi : J_1 \to J_2$ and $\phi : K_1 \to K_2$} and define the set:
\begin{equation}\label{F_1}
   	F_1 = \{ x+y +\phi(x)+\psi(y):\ x\in K_1,y \in J_1\} \oplus E \subseteq \ZZ^m \times \RR^n
\end{equation}
which is clearly a bounded measurable set since $E$ is. To complete our proof we need to show that  $F_1$ is packing $\ZZ^m \times \RR^n$ by translations of $L$  and $M$ separately and also show that all these finitely many copies of $E$ in (\ref{F_1}) do not overlap. In other words the sum in $\eqref{F_1}$ is indeed direct. By proving these two arguments, it is easy to see that 
\begin{equation*}
    \vol_n(F_1) = \vol_n(E) \cdot \Abs{K_1} \cdot \Abs{J_1} =
 \vol_n(E) \det(L_1) \det(M_1) = \det(L).
\end{equation*}
which implies that $F_1$ tiles $\ZZ^m \times \RR^n$ by translations of $L$  as we have to show. (Any measurable set that packs with a lattice and its volume is equal to the volume of the lattice also tiles with the lattice. This is easy to see by integrating the sum function of the translated tiles over a large volume. This is also true for any periodic arrangement: if a measurable set packs periodically and its volume times the density of the translates is 1, then this packing is in fact a tiling.)

We will show the packing condition of the set $F_1$ by translations of the lattice $L$. The arguments we will use are based on equations we showed earlier in the proof. Similar equations hold for $M$ and so the proof for the packing condition of the set $F_1$ by translations of the lattice $M$ is similar. We remind the reader that
\begin{equation*}
   	J_1\subseteq L_1,\quad J_2\subseteq L_2 '\subseteq \{0\}^m \times \RR^n 
\end{equation*}
and
\begin{equation*}
  	 K_1 \subseteq M_1, \quad K_2 \subseteq M_2' \subseteq \{0\}^m\times \RR^n.
\end{equation*}
 
Assume that there exist two points of $L$, $l= l_1 + l_2$ and $\widetilde{l}=\widetilde{l_1} + \widetilde{l_2}$ with $l_1,\widetilde{l_1} \in L_1$ and $l_2, \widetilde{l_2}\in L_2$ such that $l+F_1$ and $\widetilde{l} + F_1$ intersect on a set of positive measure. This means that there exists $x,\widetilde{x} \in K_1$, $y,\widetilde{y} \in J_1$ such that the translations  of $E$ by the elements
\begin{equation}\label{sum}
   	l_1 +l_2 +x +y + \phi(x) +\psi (y)
\end{equation}
and 
\begin{equation}\label{tildesum}
   	\widetilde{l_1} +\widetilde{l_2}+\widetilde{x}+\widetilde{y} +\phi(\widetilde{x})+\psi(\widetilde{y})
\end{equation}
overlap on a set of positive measure. We will show that $x=\widetilde{x}$, $y=\widetilde{y}$, $l_1=\widetilde{l_1}$ and $l_2 = \widetilde{l_2}$. As a consequence of this, we will get that the sum  in (\ref{F_1}) is indeed direct (i.e., the translated copies of $E$ that form the set $F_1$ are disjoint).

Notice that $l_1 +y$, $\widetilde{l_1} +\widetilde{y} \in L_1$, $x, \widetilde{x}\in K_1$ with the remaining terms in the sums (\ref{sum}) and (\ref{tildesum}) lying on $\{0\}^m \times \RR^n$. From (\ref{K_1 + L_1}) we get that the set $L_1 \oplus \{0\}^m \times \RR^n$ forms a tiling of $\ZZ^m \times \RR^n$ by translations of $K_1$ and so $x=\widetilde{x}$. Again by (\ref{K_1 + L_1}) we get that the set $K_1 \oplus \{0\}^m\times \RR^n$  forms  a tiling of $\ZZ^m \times \RR^n$ by translations of $L_1$ which implies that 
\begin{equation}\label{l_1+y}
   	l_1 +y = \widetilde{l_1} + \widetilde{y}.
\end{equation}
From $x= \widetilde{x}$ we also get that $\phi(x)=\phi(\widetilde{x})$ and so, after removing $l_1+y$ and $\wt l_1 + \wt y$ from (\ref{sum}) and (\ref{tildesum}) and also $\phi(x)$ and $\phi(\wt x)$, the translations $l_2+\psi(y)$ and $\widetilde{l_2} + \psi(\widetilde{y})$ of $E$ must overlap. But $l_2 + \psi(y), \widetilde{l_2} + \psi(\widetilde{y})\in L_2 '$ and we have showed that $L_2' \oplus E$ is a packing and so $l_2 + \psi(y)= \widetilde{l_2} + \psi (\widetilde{y})$.  Since
$L_2 ' = L_2 \oplus J_2$ and $l_2, \widetilde{l_2}\in L_2$, $\psi(y),\psi(\widetilde{y})\in J_2$ we obtain that $l_2=\widetilde{l_2}$ and $\psi(y)=\psi(\widetilde{y})$ which also implies that $y=\widetilde{y}$. Finally from (\ref{l_1+y}) we have that $l_1= \widetilde{l_1}$ as we had to show and this completes our proof that $F_1$ packs with $L$. The same proof, with $M$ in place of $L$, shows that $F_1$ packs with $M$ as well and our proof is complete.

\end{proof}

\bibliographystyle{alpha}
\bibliography{mk-bibliography.bib}

\end{document}